\documentclass[11pt]{amsart}
\usepackage[margin=1in]{geometry}
\numberwithin{equation}{section}
\newtheorem{thm}{Theorem}[section]
\newtheorem{prop}[thm]{Proposition}
\newtheorem{lem}[thm]{Lemma}

\newtheorem{defn}[thm]{Definition}

\newtheorem{clm}[thm]{Claim}

\newcommand{\ut}{u_\tau}
\newcommand{\pst}{\psi_\tau}
\newcommand{\rt}{\rho_\tau}
\newcommand{\RN}{\mathbb{R}^N}
\newcommand{\sk}{\psi_k }
\newcommand{\des}{\Delta e^\psi }
\newcommand{\es}{e^\psi }

\newcommand{\skx}{\psi_k (x)}
\newcommand{\esk}{e^{\psi_k }}

\newcommand{\esbj}{e^{\bar{\psi}_j}}

\newcommand{\uk}{u_k }

\newcommand{\uko}{u_{k-1} }

\newcommand{\utj}{\tilde{u}_j }

\newcommand{\ubj}{\bar{u}_j }
\newcommand{\sbj}{\bar{\psi}_j }

\newcommand{\omt}{\Omega_T}
\newcommand{\io}{\int_\Omega}

\newcommand{\po}{\partial\Omega }
\newcommand{\te}{\theta_\varepsilon }
\newcommand{\edu}{e^{-\plap}}
\newcommand{\plap}{\Delta_pu}
\newcommand{\plapt}{\Delta_p\ut}
\newcommand{\plapk}{\Delta_pu_k}
\newcommand{\plapj}{\Delta_p\ubj}

\begin{document}
	\title[A Crystal Surface Model]{Existence theorems for a  crystal surface model involving the p-Laplace operator}
	\author{Xiangsheng Xu}\thanks
	%\address
	{Department of Mathematics and Statistics, Mississippi State
		University, Mississippi State, MS 39762.
		{\it Email}: xxu@math.msstate.edu. To appear in SIAM J. Math. Anal..}
	\begin{abstract} The manufacturing of crystal films lies at the heart of modern nanotechnology. How to accurately predict the motion of a crystal surface is of fundamental importance. Many continuum models have been developed for this purpose, including a number of PDE models, which are often obtained as the continuum limit of a family of kinetic Monte Carlo models of crystal surface relaxation that includes both the solid-on-solid and discrete Gaussian models. In this paper we offer an analytical perspective into some of these models. To be specific, we study the existence of a weak solution to the boundary value problem for the equation $ - \Delta e^{-\mbox{div}\left(|\nabla u|^{p-2}\nabla u\right)}+au=f$, where $p>1, a>0$ are given numbers and $f$ is a given function. This problem is derived from  a crystal surface model proposed by J.L.~Marzuola and J.~Weare (2013 Physical Review,  E 88, 032403). The mathematical challenge is due to the fact that the principal term in our equation is an exponential function of a p-Laplacian.  Existence of a suitably-defined weak solution is established under the assumptions that $p\in(1,2], \ N\leq 4$, and $f\in W^{1,p}$. Our investigations reveal that the key to our existence assertion is how to control the set where $-\mbox{div}\left(|\nabla u|^{p-2}\nabla u\right)$ is $\pm\infty$.
		
		\end{abstract}
	\maketitle

	\section{Introduction}

Let $\Omega$ be a bounded domain in $\RN$ with smooth boundary $\po$. Given that $p>1$,  $a>0$, and a function $f=f(x)$, we consider the boundary value problem 
%We have  
\begin{eqnarray}
-\Delta e^{-\plap} +a u&=&f\ \ \mbox{in $\Omega$,}\label{sta1}\\
%e^{-\plap}&=&\rho\ \ \mbox{in $\Omega$,}\\
\nabla u\cdot\nu&=&\nabla e^{-\plap}\cdot\nu=0\ \ \mbox{on $\po$,}\label{sta2}
\end{eqnarray}	
where  $\Delta_p$ is the $p$-Laplacian, i.e., $\plap=\mbox{div}\left(|\nabla u|^{p-2}\nabla u\right)$, and $\nu$ is the unit outward normal to $\po$.

Our interest in the problem originates in the mathematical description of the evolution of a crystal surface. The surface of a crystal below the roughing temperature consists of steps and terraces. According to the Burton, Cabrera and Frank (BCF) model \cite{BCF}, atoms detach from the steps, diffuse across terraces, and reattach at new locations, inducing an overall evolution of the crystal surface. At the nanoscale, the motion of steps is described by large systems of ordinary differential equations for step positions (\cite{AKW}, \cite{GLL2}). At the macro-scale, this description is often reduced conveniently to nonlinear PDEs for macroscopic variables such as the surface height and slope profiles (see \cite{K, GLL2} and the references therein).
%, see, e.g., (\cite{BCF}, \cite{PV}). Nanoscale materials hold the promise of leading to breakthroughs in the development of electronics.

To see the connection between our problem \eqref{sta1}-\eqref{sta2} and certain existing continuum models, we first observe from the conservation of mass that the dynamic equation 
%Let us begin by reviewing the continuum model with respect to 
for the surface height profile $u(t, x)$ of a solid film is governed by
\begin{equation}
\partial_t u + \mbox{div} J = 0,
\end{equation}
where $J$ is the adatom flux.
%On the other hand, by virtue 
%In view of 
By Fick's law \cite{MK}, $J$ can be written as
$$J = -M(\nabla u)\nabla\rho_s.$$
Here   $M(\nabla u)$ is the mobility 
%a functional of the gradient
%of $u$ 
and $\rho_s$ is the local equilibrium density of adatoms. On account of the Gibbs-Thomson relation \cite{KMCS, RW, MK}, which is connected to the theory of molecular capillarity, the corresponding local equilibrium
density of adatoms is determined by
$$
\rho_s=\rho_0e^{\frac{\mu}{kT}},$$
where $\mu$ is the chemical potential, $\rho_0$ is a constant reference density, $T$ is the temperature, and $k$ is the Boltzmann constant. Denote by $\Omega$
the ``step locations area" of interest.  Then we can take the general surface energy $G(u)$ to be
\begin{eqnarray}
G(u)=\frac{1}{p}\io|\nabla u|^pdx,\ \ p\geq 1.
\end{eqnarray}
The justification for this, as observed in \cite{MW}, is that it can retain many of the interesting features of the microscopic system that are lost in the more standard scaling regime. The chemical potential $\mu$ is
defined as the change per atom in the surface energy. That is,
%can be expressed as
\begin{equation}
\mu=\frac{\delta G}{\delta u}=-\plap.
%\mbox{div}\left(|\nabla u|^{p-2}\nabla u\right).
\end{equation}
%\beta_1\frac{\nabla u}{|\nabla u|}+\beta_2
%The conservation of mass dictates that the evolution  of the surface height of a
%solid film $u(t, x)$ follows
After incorporating
those physical parameters into the scaling of the time and/or spatial variables \cite{GLL3,LMM}, we can rewrite the
evolution equation for $u $  as
\begin{equation}\label{exp}
\partial_t u=\mbox{div}\left(M(\nabla u)\nabla e^{\frac{\delta G}{\delta u}}\right).
\end{equation}
In the diffusion-limited (DL) regime, where the dynamics is dominated
by the diffusion across the terraces and  $M \equiv 1$
%Now we consider 
%the original exponential model \eqref{exp} 
the above equations reduces to 
%in DL regime
\begin{equation}
\partial_t u=\mbox{div}\left(\nabla e^{\frac{\delta G}{\delta u}}\right)\label{p1}
%\end{eqnarray}
=\Delta e^{-\plap}.
\end{equation}
%with surface energy $G=\io\frac{1}{p}|\nabla u|^pdx, \ p\geq 1$, and thus $\plap$  is the p-Laplacian of $u$, 
%i.e., $\plap=\mbox{div}\left(|\nabla u|^{p-2}\nabla u\right)$. %$a_0>0$ is a given number, 
%and. 
%The physical explanation of the p-Laplacian
%surface energy can be found in \cite{MW}. From the atomistic scale of solid-on-solid (SOS) model, the
%transitions between atomistic configurations are determined by the number of bonds that each
%atom would be required to break in order to move. It is worth noting that for $p = 1$ and $ N=1$ \cite{LMM} developed an
%explicit solution to characterize the dynamics of facet positions.
%, which is also
%confirmed by numerical simulation.
%Let . We consider 
This equation is assumed to hold in a space-time domain  $\omt\equiv \Omega\times(0,T)$, $ T>0$,
%Consequently, we are led to the study of the equation
%the value problem
%\begin{eqnarray}
%\partial_t u &=& \Delta \edu \ \ \ \mbox{in $\omt\equiv \Omega\times(0,T)$,}\label{p1}
%\end{eqnarray}
%where  $\Omega$ is a bounded domain in $\mathbb{R}^N$ with $C^2$ boundary $\partial\Omega$, $T>0$.
%We shall 
coupled with  the following  initial boundary conditions
\begin{eqnarray}
\nabla u\cdot\nu=\nabla \edu\cdot\nu &=& 0 \ \ \ \mbox{ on $\Sigma_T\equiv \partial\Omega\times(0,T)$},\label{p2}\\
u(x,0)&=& u_0(x) \ \ \ \mbox{on $\Omega$.}\label{p3}
\end{eqnarray}
%Here  $\nu$ is the unit outward normal vector to the boundary.
As we shall see, a priori estimates for this problem are rather weak. As a result, an existence theorem seems to be hopeless. Instead, we focus on the associated stationary problem. That is, we discretize the time derivative in \eqref{p1}, thereby obtaining the following stationary equation 
\begin{equation}
\frac{u-v}{\delta}-\Delta e^{-\plap}=0\ \ \mbox{in $\Omega$.}
\end{equation}
Here $v$ is a given function. Initially, $v=u_0(x)$. The positive number $\delta$ is the step size. Set $a=\frac{1}{\delta}$  and $f=\frac{1}{\delta}v$. This leads to the boundary value problem \eqref{sta1}-\eqref{sta2}. 

%where .
The objective of this paper is to establish an existence assertion for the stationary problem problem \eqref{sta1}-\eqref{sta2}, while the time-dependent problem \eqref{p1}-\eqref{p3} is left open.

If we linearize the exponential term
\begin{equation}
e^{-\plap}\approx 1-\plap,
\end{equation}
then \eqref{p1} reduces to
\begin{equation}
\partial_tu+\Delta\plap=0.
\end{equation}
Giga-Kohn \cite{GK} proved that there is a finite time extinction for the equation when $p > 1$. For
the difficult case of $p = 1$, Giga-Giga \cite{GG} developed an $H^{-1}$ total variation gradient
flow to analyze this equation and they showed that the solution may instantaneously develop jump discontinuity in the explicit example of important crystal facet
dynamics. This explicit construction of the jump discontinuity solution for facet
dynamics was extended to the exponential PDE in \cite{LMM}.

%However, t
The time-dependent problem in the case where $p=2$ has been investigated in \cite{LX}. The mathematical novelty there is that the exponent $-\Delta u$ is only a measure. But the singular part of the measure is such that the composition $e^{-\Delta u}$ is still a well-defined function. A gradient flow approach to the problem can be found in \cite{GLL3}. We also would like to mention two other related articles \cite{GLLX, LX2}. Note that if $p=2$ then the principal term in \eqref{sta1}, i.e., $e^{-\Delta u}$, can be viewed as  a monotone operator in  a suitable function space. This property is essential to the results in \cite{LX,GLL3}. If $p\ne 2$, this property is no longer  true. Moreover, the exponent becomes nonlinear. Subsequently, we lose most of the a priori estimates in \cite{LX}. What remains is
%The remaining ones are 
collected in the following
%In addition,
%\eqref{p1} can no longer be viewed as a gradient flow. 
%In spite of these, we still have the following
%To see what is still available to us from \eqref{p1}-\eqref{p3}, we proceed to perform some formal analysis. That is, we assume that 
\begin{lem}\label{lapri}
	If $u$ is a classical solution of \eqref{p1}-\eqref{p3}, then we  have
	\begin{eqnarray}
	\frac{1}{p}\int_{\Omega}|\nabla u(x,s)|^p \, dx
	+4\int_{\Omega_s}\left|\nabla \sqrt{\rho}\right|^2 \, dx\, dt
	&=&	\frac{1}{p}\int_{\Omega}|\nabla u_0(x)|^p\, dx,\label{nm2}\\
	\int_{\Omega}\ln\rho \, dx&=&0,\label{nm3}\\
	\io u(x,t)dx&=&\io u_0(x)dx.\label{nm5}
	\end{eqnarray}
	where $s>0$, $\Omega_{s}=\Omega\times(0,s)$, and
	\begin{equation}\label{nm6}
	\rho=\edu.
	\end{equation}
	% $$$$ 
\end{lem}
% in our subsequent calculations. 
%Note that
\begin{proof} We calculate
	\begin{eqnarray}
	\io\plap\partial_tudx&=&-\io|\nabla u|^{p-2}\nabla u\nabla \partial_tudx=-\frac{1}{p}\frac{d}{dt}\io|\nabla u|^pdx,\\
	\int_{\Omega}\Delta\edu\cdot\plap \, dx
	&=&-\int_{\Omega}\nabla \edu\cdot\nabla\plap\, dx\nonumber\\
	&=&\int_{\Omega}\edu\left|\nabla\plap\right|^2 \, dx\nonumber\\
	&=&4\int_{\Omega}\left|\nabla e^{-\frac{1}{2}\plap}\right|^2 \, dx.
	\end{eqnarray}
	Multiply through \eqref{p1} by $\plap$ and integrate the resulting equation with respect to the space variables over
	$\Omega$ to obtain
	%Multiply through \eqref{p1} by $\plap$ and integrate the resulting equation over $\Omega$ to obtain
	%	Take the gradient of both sides of \eqref{p1}, take the dot product of the resulting equation with $\nabla u$, and then
	%	integrate over $\Omega$ to obtain
	\begin{equation}\label{r11}
	\frac{1}{p}\frac{d}{dt}\int_{\Omega}|\nabla u(x,t)|^p\, dx
	+4\int_{\Omega}\left|\nabla e^{-\frac{1}{2}\plap}\right|^2 \, dx=0.
	\end{equation}
	%Here we have used the fact that 
	Integrate \eqref{r11} with respect to $t$
	to arrive at \eqref{nm2}.
	By \eqref{nm6}, we have
	\begin{equation}
	-\plap =\ln\rho \ \ \ \mbox{on $\Omega$.}
	\end{equation}
	%To this end, we	
	Integrate the above equation over $\Omega$ to obtain \eqref{nm3}.
	%	$$.$$
	Similarly, we can integrate \eqref{p1} over $\Omega$ to get \eqref{nm5}.
\end{proof}	

Unfortunately, this lemma is not enough for an existence assertion for problem \eqref{p1}-\eqref{p3}. To gain any further results, we are facing two main challenges. First, it does not seem possible to derive any meaningful estimates in the time variable such as estimates (6) and (9) in \cite{LX}. Second, do equations \eqref{nm2} and \eqref{nm3} really imply that $\rho\in L^{q}(\omt)$ for some $q\geq 1$ in the context here? 
%Later we shall see that 
Obviously, the two  are interconnected.
%These two technical difficulties have prevented us from establising an existence theorem for the probelm. 
%To circumvent the first one,   
In the stationary problem \eqref{sta1}-\eqref{sta2},
of course, the first challenge mentioned earlier goes away, but the second one remains. Thus the main  mathematical interest of problem \eqref{sta1}-\eqref{sta2} is how to suitably interpolate between $\ln \rho$ and $\nabla\sqrt{\rho}$.	 We must point out that condition \eqref{nm3} is rather weak. Indeed, we can easily construct a sequence $\{f_j\}$ such that
\begin{eqnarray}
f_j&\rightarrow& \infty \ \ \mbox{a.e. on $\Omega$ as $j\rightarrow \infty$, and}\\
\io f_j dx &=& 0.
\end{eqnarray}
For example, take $\Omega=(0,1)$ and define
$$f_j(s)=\left\{\begin{array}{ll}
j & \mbox{if $0\leq s<\frac{1}{j}$,}\\
j-4j^3(s-\frac{1}{j})&\mbox{if $\frac{1}{j}\leq s<\frac{3}{2j}$,}\\
-2j^2+j+4j^3(s-\frac{3}{2j})&\mbox{if $\frac{3}{2j}\leq s<\frac{2}{j}$,}\\
j&\mbox{if $\frac{2}{j}\leq s\leq 1$.}
\end{array}\right.
$$
Note that $f_j$ is continuous and piecewise linear and satisfies the boundary condition \eqref{sta2}. This means that equation \eqref{nm3} cannot be approximated. On the other hand, 
%This example . Of course, 
we obviously can not prevent a sequence with the boundary condition \eqref{sta2} from going to infinity if we only have some control on its partial derivatives.  That is to say, neither  \eqref{nm2} nor \eqref{nm3}  alone is sufficient for our purpose, and we must find a right combination between the two and equation \eqref{sta1}. This constitutes the core of our mathematical analysis.
%of the two conditions  and is not sufficient 
%The main objective of this paper is to 
%Nonetheless,
%Roughly speaking, we wll show that without the first challenge the second one becomes manageable. 
Our investigations reveal that  the set where $\plap$ is negative infinity and the set where it is positive infinity  play two significantly different roles with the former commanding most of our attention, while the latter is similar to the case already considered in \cite{LX,LX2}. %Obviously, $\plap$ is infinity on both sets.

In view of Lemma \ref{lapri} and the analysis in \cite{LX},	we can give the following definition of a weak solution.
\begin{defn}
	We say that a pair $(u,\rho)$ is a weak solution to \eqref{sta1}-\eqref{sta2} if the following conditions hold:
	\begin{enumerate}
		\item[(D1)]
		$\rho\in W^{2,p^*}(\Omega)$,
		$u\in W^{1,p}(\Omega)$,  
		$\plap\in \mathcal{M}(\overline{\Omega})\cap \left(W^{1,p}(\Omega)\right)^*$, where $p^*=\frac{Np}{N-p}$, $\left(W^{1,p}(\Omega)\right)^*$ is the dual space of $W^{1,p}(\Omega)$,  and $\mathcal{M} (\overline{\Omega})$ is the space of bounded Radon measures on $\overline{\Omega}$;
		\item[(D2)] 
		%According to , we have
		%The term
		%$- \plap$ has a decomposition
		Let
		$$-\plap=g_a+\nu_s$$
		be  the Lebesgue decomposition of $-\plap$ (\cite{EG}, p.42). That is, $g_a\in L^1(\Omega)$ and the support of $\nu_s\equiv A_0$ has Lebesgue measure $0$.
		%with respect to the Lebesgue measure in the sense of ,  
		%where $g_a$ is the absolutely continuous part of $-\plap$ and $\nu_s$ is the singular part. That is, .  
		Then  there holds
		\begin{equation}
		\rho=e^{\, g_a}\ \ \ \mbox{a.e. on $\Omega$;}\label{ns1}
		%\setminus A_0$, and $\rho=0$ on $A_0$.}
		\end{equation}
		% in the decomposition of $-\Delta u$ with respect to the %Lebesgue measure ;
		% as in the Lebesgue Decomposition Theorem 
		%given that
		% let
		% \begin{equation}
		% -\Delta u=g_a+\nu_s
		%\end{equation} be , and $\nu_s$ is the singular part. Then 
		% where $A_0$ is the support of $\nu_s$; 
		\item[(D3)] 
		%there hold the equations
		We have
		\begin{eqnarray}
		- \Delta \rho+au &=&f \ \ \ \mbox{a.e. on $\Omega$,}\label{ow11}\\
		\nabla \rho\cdot\nu &=& 0 \ \ \ \mbox{a.e. on $\po$}.
		\end{eqnarray}
		%	where $g_a$ is the absolutely continuous part of $-\Delta u$ with respect to
		%	the Lebesgue measure in the Lebesgue Decomposition Theorem;
		The boundary condition $\nabla u\cdot\nu=0$ on $\po$ is satisfied in the sense
		$$ \langle -\plap, \xi \rangle
		=\io|\nabla u|^{p-2}\nabla u \cdot \nabla\xi \, dx \ \ \ \mbox{for all $\xi\in W^{1,p}(\Omega)$,}$$
		where $\langle \cdot,\cdot \rangle$ is the duality pairing between $W^{1,p}(\Omega)$ %$L^2\left(0,T;\left(W^{1,p}(\Omega)\right)^*\right)$
		and $\left(W^{1,p}(\Omega)\right)^*$
		% $L^2\left(0,T;W^{1,2}(\Omega)\right)$.
		.
	\end{enumerate}
\end{defn}

%We would like to refer the reader to \cite{LX} for some remarks  about the definition. 
%In particular
An example in \cite{LX} shows that the singular part in $\plap$ is an intrinsic property of our solutions.  Physically, the singularities represent rupture defects and pinning on the surface evolution
%	These defects are always point downward and almost stationary in time
due to the asymmetric in the exponential curvature dependent mobility.
The pinning point ruptures in the epitaxial growth models
% distinctive feature 
was carefully studied numerically in \cite{MW}.	An easy way to remove the singular part in $-\plap$ is by adding a lower order perturbation to the equation \eqref{sta1}.
To be precise, we consider the problem
\begin{eqnarray}
-\Delta e^{-\plap} +\varepsilon\plap +au& =&f \ \ \ \mbox{in $\Omega$}\label{ota20}\\
\nabla u\cdot\nu=\nabla e^{-\plap}\cdot\nu &=& 0 \ \ \ \mbox{ on $\po$},\label{pp2}
\end{eqnarray}
%with the same initial and boundary conditions as before, 
where $\varepsilon>0$ is a small perturbation parameter. In this case, we will have $\plap\in L^2(\Omega)$. Indeed, we use $-\plap$ as a test function in \eqref{ota20} to obtain
\begin{equation}
\int_{\Omega}|\nabla u|^p\, dx
+\int_{\Omega}|\nabla e^{-\frac{1}{2}\plap}|^2\, dx
+\varepsilon\int_{\Omega}(\plap)^2\, dx\leq c\int_{\Omega}|\nabla f|^p\, dx.
\end{equation}

%Here we only mention that

Our main result is the following 
\begin{thm}\label{th1.1}	Assume that $\Omega$ is a bounded domain in $\mathbb{R}^N$ with $C^{2,\alpha}$ boundary for some $\alpha\in (0,1)$, $N\leq 4$, $a>0$,  and $f\in  W^{1,p}(\Omega)$ with $1<p\leq 2$. Then there is a %unique 
	weak solution to \eqref{sta1}-\eqref{sta2}. 
	%Furthermore, the uniqueness assertion holds for these solutions whose Laplacians have no singular parts. That is, there is only
	%	one weak solution $u$ to \eqref{sta1}-\eqref{sta2} with $\nu_s=0$.
\end{thm}

We have not considered the case where $p=1$. The physical relevance of this case  can be found in \cite{KDM}. It is also %interesting because it is 
related to the motion by surface curvature. Our key compactness result Claim \ref{gup} relies on (ii) in Lemma \ref{plap}, which fails when $p=1$. 
%Many of our results fail when $p=1$, which can easily be seen from our proof. 
%We will point this out along the way. 
Thus it would be interesting to know if we can  take the limit of our solutions as $p\rightarrow 1$. Theorem \ref{th1.1} should also hold for $p> 2$.  In the remark following Claim \ref{posi} below we shall see why we have to require $p\leq 2$. Since we allow $p$ to be arbitrarily close $1$, the Sobolev embedding theorem forces us to impose the condition $N\leq 4$.
%We leave that to the interested reader.
The uniqueness assertion for problem \eqref{sta1}-\eqref{sta2} is still open. The difficulty here is due to the fact that the operator
$-\Delta e^{-\plap}$ does not seem to be monotone anymore for $p\ne 2$.

A solution to \eqref{sta1}-\eqref{sta2} will be constructed as the limit of a sequence of approximate solutions. The key is to design an approximation scheme that can generate sufficiently regular approximate solutions so that all the preceding formal calculations are made vigorous. Then we must be able to show that the sequence of approximate solutions does not converge to infinity a.e. on $\Omega$. This is accomplished in Section 3.
%  This paper is organized as follows. 
In Section 2 we state a few preparatory lemmas, while in Section 4 we make some further remarks about the time-dependent problem.

Finally, we make some remarks about the notation. The letter $c$ denotes a positive constant. In theory, its value can be computed from various given data. In the applications of the Sobolev embedding theorems, whenever the term $N-2$ appears in a denominator, it is understood that $N>2$ because the case where $N=2$ can always be handled separately.
% and present our approximate 
%problem. The existence of a classical solution is established for the problem. We form a sequence of approximate solutions based upon implicit discretization in the time variable. Section 3 is devoted to the proof that the estimates \eqref{nm1} -\eqref{nm4} are all preserved for the sequence, and this is enough to justify passing to the limit.
\section{Preliminaries}
In this section we state a few preparatory lemmas.

Relevant interpolation inequalities for Sobolev spaces are listed in the following lemma.
\begin{lem}\label{linterp}	Let $\Omega$ be a bounded domain in $\mathbb{R}^N$. Denote by
	$\|\cdot\|_p$ the norm in the space $L^p(\Omega)$. Then we have:
	\begin{enumerate}
		\item $ \|f\|_q\leq\varepsilon\|f\|_r+\varepsilon^{-\sigma} \|f\|_p$, where
		$\varepsilon>0, p\leq q\leq r$, and $\sigma=\left(\frac{1}{p}-\frac{1}{q}\right)/\left(\frac{1}{q}-\frac{1}{r}\right)$;
		\item If $\po$ is Lipschitz, then for each $\varepsilon >0$ and each $q\in (1, p^*)$, where $p^*=\frac{pN}{N-p}$ if $N>p\geq 1$ and any number bigger than $p$ if $N=p$, there is a positive number $c=c(\varepsilon, p)$ such that
		\begin{eqnarray}
		\|f\|_q&\leq &\varepsilon\|\nabla f\|_p+c\|f\|_1\ \ \mbox{for all $f\in
			W^{1,p}(\Omega)$}.\label{otn9}
		\end{eqnarray}
		\item If $\po$ is $C^{2, \alpha}$ for some $\alpha\in (0,1)$ and $q$ is given as in (2), then
		\begin{eqnarray}
		\|\nabla g\|_q&\leq &\varepsilon\|\Delta g\|_p+c\|g\|_1\ \ \mbox{for all $g\in
			W^{2,p}(\Omega)$}.	\label{otn10}
		\end{eqnarray}
	\end{enumerate}
\end{lem}
%Our first result is an elementary inequality whose
%Some background information on
This lemma can be found in \cite{GT,G,PS}.
The next lemma collects a few frequently used elementary inequalities.
% in the following lemma.
\begin{lem}\label{elmen} For $x,y\in \mathbb{R}^N$ and $\ a, b\in \mathbb{R}^+$, we have:
	\begin{enumerate}
		\item[(4)] $|x|^{p-2}x\cdot(x-y)\geq \frac{1}{p}(|x|^p-|y|^p);$
		\item[(5)] 
		% Let $a\geq 0, \ b\geq 0$. Then we have
		%There holds
		%Let $a, b\in \mathbb{R}^+$. Then we have
		%	\begin{eqnarray}
		%(a+b)^\alpha&\leq&a^\alpha+b^\alpha\ \ \mbox{if $0<\alpha\leq 1$},\\
		%	(a+b)^\alpha&\leq&2^{\alpha-1}(a^\alpha+b^\alpha)\ \ \mbox{if $\alpha>1$},\\
		$ab\leq \varepsilon a^p+\frac{1}{\varepsilon^{q/p}}b^q \ 
		\mbox{if $\varepsilon>0,\, p, \, q>1$ with $\frac{1}{p}+\frac{1}{q}=1$}.$
		%	\end{eqnarray}
		%if $f$ is an increasing function on $\mathbb{R}$ and $F$ an anti-derivative of $f$, then
		%	$$f(s)(s-t)\geq F(s)-F(t).
		%\ \ \mbox{ for all $s,t\in \mathbb{R}$}.
		%	$$
	\end{enumerate}
\end{lem}
\begin{lem}\label{plap}Let $x,y$ be any two vectors in $\RN$. Then:
	\begin{enumerate}
		\item[\textup{(i)}] For $p\geq 2$,
		\begin{equation*}
		\left(\left(|x|^{p-2}x-|y|^{p-2}y\right)\cdot(x-y)\right)\geq \frac{1}{2^{p-1}}|x-y|^{p};
		\end{equation*}
		\item[\textup{(ii)}] For $1<p\leq 2$,
		\begin{equation*}
		\left(1+|x|^2+|y|^2\right)^{\frac{2-p}{2}}\left(\left(|x|^{p-2}x-|y|^{p-2}y\right)\cdot(x-y)\right)\geq (p-1)|x-y|^2.
		\end{equation*}
	\end{enumerate}
\end{lem}
The proof of this lemma is contained in (\cite{O}, p. 146-148). 

\begin{lem}\label{l21}
	Let $\Omega$ be a bounded domain in $\mathbb{R}^N$ with Lipschitz boundary $\po$. Consider the problem
	\begin{eqnarray}
	-\plap +\tau |u|^{p-2}u&=& f\ \ \textup{in $\Omega$,}\label{of1}\\
	\nabla u\cdot\nu&=&0\ \ \textup{on $\po$,}\label{of2}
	\end{eqnarray}
	where $\tau>0, \ p>1, \ f\in L^{\frac{p}{p-1}}(\Omega)$. Without loss of generality, we also assume
	\begin{equation}
	p<N.
	\end{equation} Then there is a unique weak solution $u$ to the above problem in the space $W^{1,p}(\Omega)$. 
	Furthermore,
	if $f$ also lies in the space $L^{q}(\Omega)$ with
	\begin{equation}\label{of3}
	q>\frac{N}{p},
	\end{equation}   $u$ is  bounded and we have the estimate
	\begin{equation}\label{of4}
	\|u\|_\infty \leq c\|u\|_1+c\left(\|f\|_q\right)^{\frac{1}{p-1}}.
	\end{equation}
	%	If $p\geq N$, .	
\end{lem}
\begin{proof} 
	We do not believe that the estimate \eqref{of4}  is  new. Since we cannot find a good reference for it, we shall offer a proof here. We employ a technique of iteration of $L^q$ norms originally due to Moser \cite{M}. Without loss of generality, assume
	\begin{equation}
	\|u^+\|_\infty=\|u\|_\infty.
	\end{equation}
	Set $b=(\|f\|_q)^{\frac{1}{p-1}}$. For each $s> \frac{p-1}{p}$ we use $\frac{s^p}{ps-p+1}(u^++b)^{ps-p+1}$ as a test function in \eqref{of1} to obtain
	\begin{eqnarray}
	\lefteqn{s^p\io(u^++b)^{ps-p}|\nabla u^+|^pdx}\nonumber\\
	&&+\frac{\tau s^p}{ps-p+1}\io |u|^{p-2}u(u^++b)^{ps-p+1}dx\nonumber\\
	&=&\frac{s^p}{ps-p+1}\io f(u^++b)^{ps-p+1}dx.\label{os1}
	\end{eqnarray}
	Note that
	\begin{eqnarray}
	\io|u|^{p-2} u^-(u^++b)^{ps-p+1}dx
	&=&b^{ps-p+1}\io|u|^{p-2} u^-dx.\label{os2}
	%&\leq &\io(u^++b)^{2s}dx+|\Omega|a^{2s}+b^{2s-1}\io u^-dx\nonumber\\
	%&\leq &2\io(u^++b)^{2s}dx+b^{2s-1}\io u^-dx
	\end{eqnarray}
	%Keeping this and 
	Letting $u=u^+-u^-$ in \eqref{of1}
	and integrating  the resulting equation over $\Omega$, which amounts to using $1$ as a test function in the equation, yield
	\begin{equation}
	\tau\io|u|^{p-2} u^+dx=\tau\io|u|^{p-2} u^-dx+\io fdx.
	\end{equation}
	Substitute this into \eqref{os2} to obtain
	\begin{eqnarray}
	\io |u|^{p-2}u^-(u^++b)^{ps-p+1}dx&=&b^{ps-p+1}\io |u|^{p-2}u^+dx-\frac{b^{ps-p+1}}{\tau}\io fdx\nonumber\\
	&\leq &\io(u^++b)^{ps}dx+\frac{1}{\tau}\io |f|(u^++b)^{ps-p+1}dx.\label{os3}
	\end{eqnarray}
	Keeping this in mind, we can derive from \eqref{os1} that
	\begin{eqnarray}
	\lefteqn{\io \left|\nabla(u^++b)^s\right|^pdx}\nonumber\\
	&\leq&\frac{\tau s^p}{ps-p+1}\io (u^++b)^{ps}dx+\frac{2 s^p}{ps-p+1}\io |f|(u^++b)^{ps-p+1}dx\nonumber\\
	&\leq&\frac{\tau s^p}{ps-p+1}\io (u^++b)^{ps}dx\nonumber\\
	&&+\frac{2 s^p}{ps-p+1}\|f(u^++b)^{-p+1}\|_q\left(\io (u^++b)^{\frac{psq}{q-1}}dx\right)^{\frac{q-1}{q}}\nonumber\\
	&\leq&\frac{\tau s^p}{ps-p+1}\io (u^++b)^{ps}dx+\frac{2 s^p}{ps-p+1}\left(\io (u^++b)^{\frac{psq}{q-1}}dx\right)^{\frac{q-1}{q}}\nonumber\\
	&\leq&\frac{c s^p}{ps-p+1}\left(\io (u^++b)^{\frac{psq}{q-1}}dx\right)^{\frac{q-1}{q}}, \ \ c=c(\Omega, \tau, q).\label{os4}
	\end{eqnarray}
	Here we have used the fact that
	\begin{equation}
	\|f(u^++b)^{-p+1}\|_q=\left(\io \frac{|f|^q}{(u^++b)^{(p-1)q}}dx\right)^{\frac{1}{q}}\leq \frac{\|f\|_q}{b^{p-1}}=1.
	\end{equation}
	With the aid of the Sobolev inequality, we obtain
	\begin{eqnarray}
	\lefteqn{
		\left(\io (u^++b)^{\frac{sNp}{N-p}}dx\right)^{\frac{N-p}{N}}}\nonumber\\
	&\leq&c\io \left|\nabla(u^++b)^{s}\right|^pdx+c\io (u^++b)^{sp}dx\nonumber\\
	&\leq&\frac{c s^p}{ps-p+1}\left(\io (u^++b)^{\frac{psq}{q-1}}dx\right)^{\frac{q-1}{q}}+c\io (u^++b)^{sp}dx\nonumber\\
	&\leq&\frac{c s^p}{ps-p+1}\left(\io (u^++b)^{\frac{psq}{q-1}}dx\right)^{\frac{q-1}{q}}.
	\end{eqnarray}
	The last step is due to the fact that
	$$\frac{ s^p}{ps-p+1}\geq 1.
	$$
	Set $\chi=\frac{N}{N-p}/\frac{q}{q-1}$. Our assumption \eqref{of3} implies that $\chi>1$. We can write
	\eqref{os4} in the form
	\begin{eqnarray}
	\|u^++b\|_{\frac{psq\chi}{q-1}}&\leq &\left(\frac{c }{ps-p+1}\right)^{\frac{1}{ps}}s^{\frac{1}{s}}\|u^++b\|_{\frac{psq}{q-1}}\nonumber\\
	&\leq &c ^{\frac{1}{s}}s^{\frac{1}{s}}\|u^++b\|_{\frac{psq}{q-1}},\ \ \mbox{provided that $s\geq 1$.}
	\end{eqnarray}
	In view of the proof in (\cite{GT}, p. 190), we take $s=\chi^m, m=0,1,2,\cdots$ in the above inequality. Iterating and taking the limit lead to 
	\begin{equation}
	\|u^++b\|_{\infty}\leq c\|u^++b\|_{\frac{pq}{q-1}},
	\end{equation}
	whence by the interpolation inequality (1) in Lemma \ref{linterp} we have 
	\begin{equation}
	\|u^++b\|_{\infty}\leq c\|u^++b\|_{1}.
	\end{equation}
	This implies the desired result.
\end{proof} 
Further regularity results for solutions to equations of p-laplace type can be found in \cite{AZ, T} and the references therein.
%	We refer the reader to \cite{X3} for some background information on this lemma. 
%We only mention that (iii) first appeared in \cite{GST}.

Our existence theorem is based upon the following fixed point theorem, which is often called the Leray-Schauder Theorem (\cite{GT}, p.280).
\begin{lem}
	Let $B$ be a map from a Banach space $\mathcal{B}$ into itself. Assume:
	\begin{enumerate}
		\item[(H1)] $B$ is continuous;
		\item[(H2)] the images of bounded sets of $B$ are precompact;
		\item[(H3)] there exists a constant $c$ such that
		$$\|z\|_{\mathcal{B}}\leq c$$
		for all $z\in\mathcal{B}$ and $\sigma\in[0,1]$ satisfying $z=\sigma B(z)$.
	\end{enumerate}
	Then $B$ has a fixed point.
\end{lem}
\begin{lem}\label{poin}
	Let $\Omega$ be a bounded domain in $\RN$ with Lipschitz boundary and $1\leq p<N$.
	Then there is a positive number $c=c(N)$ such that
	\begin{equation}
	\|u-u_S\|_{p^*}\leq \frac{cd^{N+1-\frac{p}{N}}}{|S|^{\frac{1}{p}}}\|\nabla u\|_p \ \ \mbox{for each $u\in W^{1,p}(\Omega)$,}
	\end{equation}
	where $S$ is any measurable subset of $\Omega$ with $|S|>0$, $u_S=\frac{1}{|S|}\int_S udx$, and $d$ is the diameter of $\Omega$.
\end{lem}	
This lemma can be inferred from Lemma 7.16 in \cite{GT}. Also see \cite{G,PS}. It is a version of the Poincar\'{e} inequality.

\section{Proof of Theorem \ref{th1.1}}
In this section we first design an approximation scheme for problem \eqref{sta1}-\eqref{sta2}. Then we obtain a weak solution by passing to the limit in our approximate problems.

%Now we are ready to present our approximate problems.	
Following \cite{LX}, we introduce a new unknown function
\begin{equation}
\psi=-\plap.
\end{equation}
Then regularize this equation by adding the term $\tau |u|^{p-2}u,\ \tau>0$, to its right-hand side.  This is due to the Neumann boundary condition in our problem. By the same reason, we add $\tau\psi$ to \eqref{sta1}.
%Subsequently, we  discretize the time derivative in \eqref{p1}, thereby turning
%a parabolic problem into an elliptic one. 
This leads to the study of the system 
%Consider the problem
\begin{eqnarray}
-\des+\tau\psi+a u &=&f\ \ \ \mbox{in $\Omega$},\label{ot1}\\
-\plap +\tau|u|^{p-2} u &=&\psi \ \ \ \mbox{in $\Omega$}
\end{eqnarray}
coupled with the boundary conditions
\begin{equation}
\nabla u\cdot\nu=\nabla\es\cdot\nu=0\ \ \ \mbox{on $\partial\Omega$},\label{ot2}
\end{equation}
where we assume
\begin{equation}
f\in L^\infty(\Omega. \label{ot3}
\end{equation}  
%is a given function. 
This is our approximating problem.
Basically, we have transformed a fourth-order equation into a system of two second-order equations. A mathematical motivation behind this construction is given in \cite{LX}.

\begin{thm}\label{p21}
	%Let \eqref{ot3} hold, and assume that
	Let $\Omega$ be a bounded domain in $\mathbb{R}^N$ with $C^{2,\alpha}$ boundary with some $\alpha\in(0,1)$, and assume that $1<p<N$ and \eqref{ot3} hold. Then there is a weak solution $(\psi, u)$ to \eqref{ot1}-\eqref{ot2} with 
	\begin{eqnarray}
	\psi&\in& W^{2,q}(\Omega)\ \ \textup{for each $q>1$},\label{reg1}\\
	u&\in& C^{1,\lambda}(\overline{\Omega}),\ \ \lambda\in (0,1).\label{reg2}
	\end{eqnarray}.
\end{thm}
\begin{proof}
	The existence assertion will be established via the Leray-Schauder Theorem. For this purpose, we define an operator $B$ from $L^\infty(\Omega)$
	into itself as follows: for each $g\in L^\infty(\Omega)$ we say $B(g)=\psi$ if $\psi$ is the unique solution of the linear boundary value problem
	\begin{eqnarray}
	-\mbox{div}\left(e^g\nabla\psi\right)+\tau\psi &=&f-au\ \ \ \mbox{in $\Omega$},\label{om3}\\
	\nabla\psi\cdot\nu&=&0\ \ \ \mbox{on $\partial\Omega$},\label{om4}
	\end{eqnarray}
	where $u$ solves the problem
	\begin{eqnarray}
	-\plap +\tau |u|^{p-2}u &=&g\ \ \ \mbox{in $\Omega$},\label{om5}\\
	\nabla u\cdot\nu&=& 0\ \ \ \mbox{on $\partial\Omega$}.\label{om6}
	\end{eqnarray}
	Concerning 	the preceding boundary value problem, a theorem in (\cite{O}, p.124) asserts that the problem  has a weak solution $u$ in the space $W^{1,p}(\Omega)$. Obviously, the uniqueness of such a solution is a consequence of Lemma \ref{plap}. In fact, we can further conclude from \cite{AZ,L,T} that
	$u$ satisfies \eqref{reg2}.
	Observe that since $g\in L^\infty(\Omega)$ the equation \eqref{om3} is uniformly elliptic. According to the classical regularity theory for linear elliptic equations,   problem \eqref{om3}-\eqref{om4} has a unique solution $\psi$ in the space $W^{1,2}(\Omega)\cap C^{0,\beta}(\overline{\Omega})$
	for some $\beta\in (0,1)$ (\cite{GT}, Chap. 8). 
	Therefore, we can conclude that $B$ is well-defined, continuous, and maps bounded sets into precompact ones. It remains to show that there is a positive number $c$ such that
	\begin{equation}
	\|\psi\|_\infty\leq c\label{ot8}
	\end{equation}
	for all $\psi\in L^\infty(\Omega)$ and $\sigma\in [0,1]$ satisfying
	$$\psi=\sigma B(\psi).$$
	This equation is equivalent to the boundary value problem
	\begin{eqnarray}
	-\des+\tau\psi &=&\sigma(f-au)\ \ \ \mbox{in $\Omega$},\label{ot9}\\
	-\plap +\tau|u|^{p-2} u &=&\psi \ \ \ \mbox{in $\Omega$},\label{ot10}\\
	\nabla u\cdot\nu=\nabla\es\cdot\nu&=& 0\ \ \ \mbox{on $\partial\Omega$}.
	\end{eqnarray}
	%The inequality \eqref{ot8} will be a consequence of the following  claim.
	\begin{clm}We have
		\begin{eqnarray}
		\|\psi\|_{q}&\leq&\frac{1}{\tau}\|f-au\|_{q}\ \ \textup{for each $q>2$, and thus}\label{c21} \\
		\|\psi\|_\infty&\leq&\frac{1}{\tau}\|f-au\|_\infty.\label{ot12}
		\end{eqnarray}
		Furthermore,
		%\begin{clm}
		\begin{equation}\label{c22}
		\|\psi\|_2\leq\frac{1}{\tau}\|f\|_2
		\end{equation}
		%	$	$	
		%\end{clm}
	\end{clm}
	\begin{proof} We just need to slightly modify the proof of Claim 2.1  in \cite{LX}. 
		%For the reader's convevience, we shall reproduce it here. 
		Let $q>2$ be given. 
		Then the function $|\psi|^{q-2} \psi$ lies in $W^{1,2}(\Omega)$ and $\nabla\left(|\psi|^{q-2} \psi\right)=(q-1)|\psi|^{q-2}\nabla\psi$. Multiply through \eqref{ot9} by
		this function and integrate the resulting equation over $\Omega$ to obtain
		\begin{eqnarray*}
			(q-1)\int_\Omega\es|\psi|^{q-2}|\nabla\psi|^2\, dx
			+\tau\int_\Omega|\psi|^{q}\, dx&=&\sigma\int_\Omega(f-au)|\psi|^{q-2} \psi \, dx\\
			&\leq &\int_\Omega|f-au||\psi|^{q-1}\, dx\\
			&\leq &\|f-au\|_{q}\|\psi\|_{q}^{q-1}.
		\end{eqnarray*}
		Dropping the first integral in the above inequality yields \eqref{c21}.
		
		Multiplying $\psi$ through \eqref{ot9}, we obtain
		\begin{equation}
		\int_\Omega\es|\nabla\psi|^2\, dx
		+\tau\int_\Omega\psi^2\, dx
		=-\sigma\int_\Omega au\psi \, dx
		+\sigma\int_\Omega f\psi \, dx.\label{ot11}
		\end{equation}
		Upon  using $u$ as a test function in \eqref{ot10}, we can derive
		$$\int_\Omega u\psi \, dx
		=\int_\Omega|\nabla u|^p\, dx
		+\tau\int_\Omega u^p \, dx\geq 0.$$
		Keeping this in mind, we deduce from \eqref{ot11} that
		$$\tau\int_\Omega\psi^2\, dx
		\leq \sigma\int_\Omega f\psi \, dx
		\leq \int_\Omega |f||\psi| \, dx
		\leq\|f\|_2\|\psi\|_2.$$
		Then \eqref{c22} follows.
	\end{proof}
	
	To continue the proof of Theorem \ref{p21}, multiply through \eqref{ot9} by $e^\psi-1$ and integrate the resulting equation over $\Omega$ to obtain
	\begin{eqnarray}\label{ot122}
	\io|\nabla e^\psi|^2dx+\tau\io\psi(e^\psi-1)dx&=&-\sigma\io au(e^\psi-1)dx+\sigma\io f(e^\psi-1)dx\nonumber\\
	&\leq&\left| \io au(e^\psi-1)dx\right|+\|f\|_\infty\io|(e^\psi-1)|dx.
	\end{eqnarray}
	In view of \eqref{otn9}, the first integral on the right-hand side in the above equation can be estimated as follows:
	\begin{eqnarray}
	\left|\io u(e^\psi-1)dx\right|&\leq&\|u\|_2\|e^\psi-1\|_2\nonumber\\
	&\leq&c\varepsilon\|\nabla e^\psi\|_2+c\|e^\psi-1\|_1, \ \ \varepsilon>0.\label{ot13}
	\end{eqnarray}
	Here we have taken \eqref{c22} into account. For each $M>0$ we have
	\begin{eqnarray}
	\io\left|(e^\psi-1)\right|dx&=&\int_{|\psi|>M}\left|e^\psi-1\right|dx+\int_{|\psi|\leq M}\left|e^\psi-1\right|dx\nonumber\\
	&\leq&\frac{1}{M}\io\psi\left(e^\psi-1\right)dx+c(M).\label{ot14}
	\end{eqnarray}
	%Use \eqref{ot3} in the last integral in , pl
	Plug \eqref{ot13} and then \eqref{ot14} into \eqref{ot122},  choose $\varepsilon$ suitably small and $M$ suitably large in the resulting inequality, thereby derive
	\begin{equation}
	\io|\nabla e^\psi|^2dx+\tau\io\psi(e^\psi-1)dx\leq c.
	\end{equation}
	This combined with \eqref{ot14} implies that $\psi$ is bounded in $L^q(\Omega)$ for each $q\geq 1$.
	Thus we apply Lemma \ref{l21} to obtain that $u$ is bounded in $L^\infty(\Omega)$. Consequently, \eqref{ot8} follows from \eqref{ot12} and \eqref{ot3}.
	%Since $\psi$ is bounded, 
	As we mentioned earlier, we can infer \eqref{reg2} from \cite{L, AZ,T}. This together with the classical Calder\'{o}n-Zygmund estimate implies \eqref{reg1}.
	%Since $\psi$ is bounded,
	%we can conclude from the classical regularity theory that $(\psi, u)$ is a classical solution. 
	The proof is complete.
\end{proof}

\begin{proof}[Proof of Theorem \ref{th1.1}]
	Without loss of genelarity, we may assume that
	\begin{equation}
	f\in L^\infty(\Omega)\cap W^{1,p}(\Omega).
	\end{equation}
	Otherwise, $f$ can be approximated by a sequence in the above space in $W^{1,p}(\Omega)$.
	We shall show that we can take $\tau\rightarrow 0$ in \eqref{ot1}-\eqref{ot2}. For this purpose we need to derive estimates that are uniform in $\tau$. We write
	\begin{equation}
	u=\ut,\ \ \psi=\pst.
	\end{equation}
	Then problem \eqref{ot1}-\eqref{ot2} becomes
	\begin{eqnarray}
	-\Delta \rt+\tau\pst+a \ut &=&f\ \ \ \mbox{in $\Omega$},\label{ot1t}\\
	e^{\pst}&=&\rt\ \ \ \mbox{in $\Omega$},\label{ot4t}\\
	-\plapt+\tau|\ut|^{p-2} \ut &=&\pst \ \ \ \mbox{in $\Omega$},\label{ot3t}\\
	\nabla \ut=\cdot\nu&=&\nabla \rt\cdot\nu=0\ \ \ \mbox{on $\partial\Omega$}.\label{ot2t}
	\end{eqnarray}
	We also view $\{\ut,\rt,\pst\}$ as a sequence in the subsequent proof. Take $\tau=\frac{1}{k}$, where $k$ is a positive integer, for example. The rest of the proof is divided into several claims.
	
\end{proof}
\begin{clm}We have
	\begin{eqnarray}
	\io|\nabla\sqrt{\rt}|^2dx+\tau\io\pst^2dx+\io|\nabla\ut|^pdx+\tau\io\ut^pdx&\leq &c,\label{rue}\\
	\|\ut\|_{W^{1.p}(\Omega)}&\leq &c.\label{ue}
	\end{eqnarray}
\end{clm}
\begin{proof}
	Use $\pst=\ln\rt$ as a test function in \eqref{ot1t} to obtain
	\begin{equation}\label{nf1}
	4\io|\nabla\sqrt{\rt}|^2dx+\tau\io\pst^2dx+a\io\ut\pst dx=\io f\pst dx.
	\end{equation}
	With the aid of \eqref{ot3t}, we evaluate the last two integrals in the above equation as follows:
	\begin{eqnarray}
	\io\ut\pst dx&=&\io|\nabla\ut|^pdx+\tau\io\ut^pdx,\label{nf2}\\
	\io f\pst dx&=&\io|\nabla\ut|^{p-2}\nabla\ut\nabla fdx+\tau\io |\ut|^{p-2} \ut f dx\nonumber\\
	&\leq &\|\nabla f\|_p\|\nabla \ut\|_p^{p-1}+\tau\|f\|_p\|\ut\|_p^{p-1}.\label{nf3}
	\end{eqnarray}
	Plug \eqref{nf2} and \eqref{nf3} into \eqref{nf1}, apply the interpolation inequality (5) in Lemma \ref{elmen} in the resulting inequality, and thereby obtain
	\begin{eqnarray}\label{nf4}
	\lefteqn{\io|\nabla\sqrt{\rt}|^2dx+\tau\io\pst^2dx+\io|\nabla\ut|^pdx+\tau\io\ut^pdx}\nonumber\\&\leq &c\io|\nabla f|^pdx+c\tau\io| f|^pdx\nonumber\\
	&\leq &c\io|\nabla f|^pdx+c\io| f|^pdx.
	\end{eqnarray}
	From here on we assume that $\tau\leq 1$. The above estimate gives \eqref{rue}.
	Integrate \eqref{ot1t} over $\Omega$ to yield
	\begin{equation}
	\left|a\io \ut dx\right|=\left|\io fdx-\tau\io\pst dx\right|\leq c.
	\end{equation}
	Subsequently, we can apply the Poincar\'{e} inequality to get
	\begin{eqnarray}
	\|\ut\|_{p^*}&\leq &\|\ut-\frac{1}{|\Omega|}\io \ut dx\|_{p^*}+\frac{1}{|\Omega|^{1-\frac{1}{p}}}\left|\io \ut dx\right|\nonumber\\
	&\leq &c\|\nabla\ut\|_{p}+\frac{1}{|\Omega|^{1-\frac{1}{p}}}\left|\io \ut dx\right|\leq c.\label{nf5}
	\end{eqnarray}
	Thus \eqref{ue} follows. The proof is complete.
\end{proof}
\begin{clm}\label{aec}
	There exists a subsequence of $\{\rt\}$, still denoted by $\{\rt\}$, such that
	\begin{equation}
	\rt\rightarrow \rho\ \ \mbox{a.e. on $\Omega$ as $\tau\rightarrow 0$.}
	\end{equation}
\end{clm}
\begin{proof}We use $\arctan\rt$ as a test function in \eqref{ot1t} to obtain
	\begin{eqnarray}
	%	\io\left|\nabla\int_{0}^{\rt}\frac{1}{\sqrt{1+s^2}}ds\right|dx=
	\io\frac{|\nabla\rt|^2}{1+\rt^2}dx
	&=&\io\left(f-\tau\pst-a\ut\right)\arctan\rt dx\nonumber\\
	&\leq &\pi\io\left|f-\tau\pst-a\ut\right|dx\leq c.
	\end{eqnarray}
	Thus $\{\arctan\rt\}$ is bounded in $W^{1,2}(\Omega)$. We can extract a subsequence of $\{\arctan\rt\}$ which converges a .e. on $\Omega$. It follows that $\rt=\tan\left(\arctan\rt\right)$ also converges a.e. along the subsequence. This completes the proof.
\end{proof}
It should be noted that at this point we cannot rule out the possibility that $\{\arctan\rt\}$  goes to $\frac{\pi}{2}$ on a large set.  Thus the limit $\rho$ may not be  finite a.e. on $\Omega$.  
\begin{clm}\label{fini}
	If $\rho$ is finite on a set of positive measure, then there is a subsequence of $\{\rt\}$ which is bounded in $L^q(\Omega)$ for each $1\leq q<\frac{N}{N-2}$.
	% we have that  at least for a subsequence.
\end{clm}
\begin{proof}
	Our assumption implies that there is a positive number $s_0$ such that the set
	\begin{equation}
	\Omega_{s_0}=\{x\in\Omega: \rho(x)\leq s_0\}
	\end{equation}
	has positive measure. According to Claim \ref{aec} and Egoroff's theorem, for each $\varepsilon>0$ there is a closed set $K\subset \Omega_{s_0}$ such that
	$| \Omega_{s_0}\setminus K|< \varepsilon$ and $\rt\rightarrow\rho$ uniformly on $K$. We take $\varepsilon=\frac{1}{2}|\Omega_{s_0}|$. Then the corresponding $K$ has positive measure. We easily conclude from the uniform convergence that there is a positive number
	$s_1>s_0$ with the property
	\begin{equation}
	\rt\leq s_1\ \ \mbox{on $K$.}
	\end{equation}
	For each $s>s_1$ we use $\left(\frac{1}{s}-\frac{1}{\rt}\right)^+$ as a test function in \eqref{ot1t} to obtain
	\begin{equation}
	\io\left|\nabla \ln^+\frac{\rt}{s}\right|^2dx=\io(f-\tau\pst-a\ut)\left(\frac{1}{s}-\frac{1}{\rt}\right)^+dx\leq \frac{c}{s}.
	\end{equation}
	Denote by $S_{\tau,s}$ the set where the function $\left(\frac{1}{s}-\frac{1}{\rt}\right)^+$  is $0$. It follows that
	\begin{equation}
	K\subset S_{\tau,s}\ \ \mbox{for all $s>s_1$ and sufficiently many $\tau>0$}.
	\end{equation}
	Thus we may apply Lemma \ref{poin} to obtain
	\begin{eqnarray}
	\ln^22|\{x\in\Omega:\rt(x)>2s\}|^{\frac{N-2}{N}}&\leq&\left[\io\left(\ln^+\frac{\rt}{s}\right)^{\frac{2N}{N-2}}dx\right]^{\frac{N-2}{N}}\nonumber\\
	&\leq&\frac{c}{|K|}\io\left|\nabla \ln^+\frac{\rt}{s}\right|^2dx\leq\frac{c}{s}.
	\end{eqnarray}
	This implies the desired result.
\end{proof}
\begin{clm}\label{posi}
	The set where $\rho$ is finite has positive measure.
\end{clm}
\begin{proof}We argue by contradiction. Suppose that the claim is false. Then we have
	\begin{equation}
	\rho =\infty\ \ \mbox{a.e. on $\Omega$.}
	\end{equation}
	For each $L>0$ we define
	\begin{equation}\label{posi1}
	\gamma_L(s)=\left\{\begin{array}{ll}
	L& \mbox{if $s>L$,}\\
	s&\mbox{if $-L\leq s\leq L$,}\\
	-L&\mbox{if $s<-L$.}
	\end{array}\right.\end{equation}
	Fix $L>1$. 
	%We easily see that
	%
	%	Consequently,
	%	\begin{equation}
	%	|\nabla \gamma_L\left((\rt-1)^+\right)|\leq |\nabla\rt|\ \ \mbox{a.e. on $\Omega$.}
	%	\end{equation}
	%	Since $\rt$ satisfies \eqref{ot1t}, the classical $W^{1, p}$ estimate for Poisson's equation \cite{GT} asserts that there is a positive number $c$, independent of $\tau$, such that
	%	\begin{equation}
	%	\|\nabla \gamma_L\left((\rt-1)^+\right)\|_p\leq 	\|\nabla\rt\|_p\leq %c\|f-\tau\pst-a\ut\|_{\frac{Np}{N+p}}\leq c.
	%	\end{equation}
	%Here we have used the fact that $\{\tau|\rt|^{p-1}\}$ is bounded in %$L^{\frac{NP}{N+p}}(\Omega)$. This is due to the interpolation inequality \eqref{otn9}.
	%By 
	Multiply through \eqref{ot1t} by $\gamma_L\left((\rt-1)^+\right)$ and integrate to obtain
	\begin{equation}
	\io\left|\nabla \gamma_L\left((\rt-1)^+\right)\right|^2dx=\io(f-\tau\pst-a\ut)\gamma_L\left((\rt-1)^+\right)dx\leq cL.
	\end{equation}
	Here we have used the fact that
	\begin{equation}
	\nabla \gamma_L\left((\rt-1)^+\right)=0\ \ \mbox{on the set where either $\rt\leq 1$ or $\rt>L+1$.}
	\end{equation}
	Now consider the sequence $\{\ln\rt\gamma_L\left((\rt-1)^+\right)\}$. It is easy to see that
	\begin{eqnarray}
	\ln\rt\gamma_L\left((\rt-1)^+\right)&\geq&0 \ \ \mbox{a.e. on $\Omega$,}\\
	\lim_{\tau\rightarrow 0}\ln\rt\gamma_L\left((\rt-1)^+\right)&=&\infty\ \ \mbox{a.e. on $\Omega$.}
	\end{eqnarray}
	By \eqref{ot4t} and \eqref{ot3t}, we have
	\begin{equation}\label{ns11}
	-\plapt+\tau|\ut|^{p-2} \ut =\ln\rt \ \ \ \mbox{in $\Omega$}.
	\end{equation}
	Use $\gamma_L\left((\rt-1)^+\right)$ as a test function in the above equation, thereby deriving
	\begin{eqnarray}
	\lefteqn{\io\ln\rt\gamma_L\left((\rt-1)^+\right)dx}\nonumber\\
	&=&\io|\nabla\ut|^{p-2}\nabla\ut\nabla\gamma_L\left((\rt-1)^+\right)dx\nonumber\\
	&&+\tau\io|\ut|^{p-2} \ut \gamma_L\left((\rt-1)^+\right) dx\nonumber\\
	&\leq &\|\nabla\gamma_L\left((\rt-1)^+\right)\|_p\|\nabla\ut\|_p^{p-1}+L\tau\io|\ut|^{p-1}dx\leq c(L).\label{pl2}
	\end{eqnarray}
	The last step is due to the assumption $p\leq 2$.
	It follows from Fatou's lemma that the left-hand side of the above inequality goes to $\infty$ as $\tau\rightarrow 0$. This gives us a contradiction. The proof is complete.
\end{proof}
We would like to make a remark about 
%that 
%inequality \eqref{pl2} is the only place where we have to impose 
the condition $p\leq 2$. Note from \eqref{c21} that
\begin{equation}\label{c211}
\tau\|\pst\|_{p^*}\leq\|f-a\ut\|_{p^*}.
\end{equation}
Thus this condition
%. This restriction 
could be avoided here if we had the estimate
\begin{equation}\label{wpe1}
\|\nabla\rt\|_p\leq c\|f-\tau\pst-a\ut\|_{\frac{Np}{N+p}}.
\end{equation}
The above inequality is valid if $\rt$ satisfies the Dirichlet boundary condition $\rt|_{\po}=0$. 
%	We doubt that  it is still true i
In our case, the right hand side of \eqref{wpe1} seems to also depend on the $L^1$-norm of $\rt$.
%for the Nuemann boundary condition.
\begin{clm}\label{lnr}
	The sequence $\{\ln\rt\}$ is bounded in $L^1(\Omega)$.
\end{clm}
\begin{proof} 
	Use the number $1$ as a test function in \eqref{ns11} to get
	\begin{equation}\label{nf6}
	\left|\io \ln\rt dx\right|=\tau\left|\io |\ut|^{p-2} \ut dx \right|\leq c\tau^{\frac{1}{p}}.
	\end{equation}
	The last step is due to \eqref{nf4}. By virtue of Claims \ref{fini} and \ref{posi}, 
	\begin{equation}\label{rb1}
	\mbox{the sequence $\{\rt\}$ is bounded in $L^q(\Omega)
		$ for each $1\leq q<\frac{N}{N-2}$. }
	\end{equation}
	We will use this for $q=1$. Aslo keeping  \eqref{nf6} in mind, we estimate
	\begin{eqnarray}
	\io|\ln\rt|dx&=&\io\ln^+\rt dx+\io\ln^-\rt dx\nonumber\\
	&=&2\io\ln^+\rt dx-\io\ln\rt dx\nonumber\\
	&\leq& 2\io\rt dx+c\tau^{\frac{1}{p}}\leq c.
	\end{eqnarray}\end{proof}
\begin{clm}\label{rtc}
	The sequence $\{\rt\}$ is precompact in $W^{1,2}(\Omega)$.
\end{clm}
\begin{proof} Note that our assumptions on $N, p$ imply
	\begin{equation}\label{nss2}
	\frac{2N}{N+2}<\frac{Np}{N-p}=p^*.
	\end{equation}Use $\rt-1$ as a test function in \eqref{ot1t} to obtain
	\begin{eqnarray}
	\io|\nabla\rt|^2dx&\leq &\io(f-a\ut)(\rt-1)dx\nonumber\\
	&\leq&\|f-a\ut\|_{\frac{2N}{N+2}}\|\rt-1\|_{\frac{2N}{N-2}}\nonumber\\
	&\leq&c\|\nabla\rt\|_2+c\|\rt-1\|_2.\label{nss1}
	\end{eqnarray}
	Here we have used the fact that $\pst(\rt-1)\geq 0$. 
	Use \eqref{rb1} and the interpolation inequality \eqref{otn9} in \eqref{nss1} to obtain
	\begin{equation}
	\io|\nabla\rt|^2dx\leq c.
	\end{equation}
	This combined with \eqref{rb1} implies that
	\begin{equation}\label{rtb11}
	\mbox{$\{\rt\}$ is bounded in $W^{1,2}(\Omega)$.}
	\end{equation}
	Thus $\{\rt\}$ is precompact in $L^q(\Omega)$ for each $q\in [0, 2^*)$.
	Set $g_\tau= f-a\ut-\tau\pst$.
	Then by \eqref{c211}, the sequence  $\{g_\tau\}$ is  bounded in $L^{p^*}(\Omega)$.
	Let $\tau_1,\tau_2\in (0,1)$. We calculate from \eqref{ot1t} that
	\begin{eqnarray}
	\io|\nabla\rho_{\tau_1}-\nabla\rho_{\tau_2}|^2dx&\leq& \io(g_{\tau_1}-g_{\tau_2})(\rho_{\tau_1}-\rho_{\tau_2})dx\nonumber\\
	&\leq& c \|g_{\tau_1}-g_{\tau_2}\|_{p^*} \|\rho_{\tau_1}-\rho_{\tau_2}\|_{\frac{Np}{Np-N+p}}\nonumber\\
	&\leq& c\|\rho_{\tau_1}-\rho_{\tau_2}\|_{\frac{Np}{Np-N+p}}.
	\end{eqnarray}
	In view of our assumptions on $N, p$, we have  that
	\begin{equation}
	\frac{Np}{Np-N+p}<2^*.
	\end{equation}
	The claim follows from the precompactness of $\{\rt\}$ in $L^q(\Omega)$ for each $q\in [0, 2^*)$.
	%This implies the desired result.
	%The condition \eqref{nss2} along with \eqref{ue} implies that $\{\ut\}$ is precompact
	%	in $L^{\frac{2N}{N+2}}(\Omega)$. 
	%The claim follows.
\end{proof}
\begin{clm}\label{gup} At least a  subsequence of $\{\nabla\ut\}$ converges a.e. on $\Omega$.
	% (Passing to a subsequence if need be).Passing to a subsequence if need be, we have that
\end{clm}
\begin{proof}
	We will show that $\{\ut\}$ is precompact in $W^{1,q}(\Omega)$ for each $q<p$. The idea behind the proof has appeared elsewhere. See, for example,  the proof of Lemma 2.2 in \cite{X4}.
	
	By \eqref{ue} and Egoroff's theorem, for each $\delta>0$ there is a closed set $E\subset \Omega$ with the properties
	\begin{eqnarray}
	|\Omega\setminus E|&\leq&\delta,\\
	\ut&\rightarrow& u\ \ \ \mbox{uniformly on $E$.}
	\end{eqnarray}
	Subsequently, we can find a positive number $K$ so that
	\begin{equation}\label{utb1}
	|\ut|\leq K \ \ \ \mbox{on $E$.}
	\end{equation} 
	For any $\varepsilon>0$ we have
	\begin{equation}
	|u_{\tau_1}-u_{\tau_2}|<\varepsilon\ \ \mbox{on $E$ for sufficiently small $\tau_1,\tau_2$.}
	\end{equation}We can derive from \eqref{ot3t}  and Lemma \ref{plap} that
	\begin{eqnarray}
	\lefteqn{	\int_{E}\left(|\nabla u_{\tau_1}|^{p-2}\nabla u_{\tau_1}-|\nabla u_{\tau_2}|^{p-2}\nabla u_{\tau_2}\right)\cdot\nabla(u_{\tau_1}-u_{\tau_2})dx}\nonumber\\
	&\leq&	\io\left(|\nabla u_{\tau_1}|^{p-2}\nabla u_{\tau_1}-|\nabla u_{\tau_2}|^{p-2}\nabla u_{\tau_2}\right)\cdot\nabla\gamma_\varepsilon(u_{\tau_1}-u_{\tau_2})dx\nonumber\\
	&=&\io\left(\ln\rho_{\tau_1}-\tau_1|u_{\tau_1}|^{p-2}u_{\tau_1}-\ln\rho_{\tau_2}+\tau_2|u_{\tau_2}|^{p-2}u_{\tau_2}\right)\gamma_\varepsilon(u_{\tau_1}-u_{\tau_2})dx\nonumber\\
	&\leq & c\varepsilon,
	\end{eqnarray}
	where $\gamma_\varepsilon$ is obtained by replacing L with $\varepsilon$ in \eqref{posi1}. Apply (ii) in Lemma \ref{plap} and \eqref{utb1} to deduce
	\begin{equation}
	\int_E|\nabla(u_{\tau_1}-u_{\tau_2})|^2dx\leq c\varepsilon.
	\end{equation}
	Thus $\{\nabla\ut\}$ is precompact in $\left(L^2(E)\right)^N$. Let $q<p$  be given.
	We estimate
	\begin{eqnarray}
	\io|\nabla(\ut-u)|^q dx&=&\int_E|\nabla(\ut-u)|^q dx+\int_{\Omega\setminus E}|\nabla(\ut-u)|^q dx\nonumber\\
	&\leq &c|\Omega\setminus E|^{1-\frac{q}{p}}+\int_E|\nabla(\ut-u)|^q dx\nonumber\\
	&\leq &c\delta^{1-\frac{q}{p}}+\int_E|\nabla(\ut-u)|^q dx.
	\end{eqnarray}
	Therefore,
	\begin{equation}
	\limsup_{\tau\rightarrow 0}\leq c\delta^{1-\frac{q}{p}}\ \mbox{at least along a subsequence.}
	\end{equation}
	Since $\delta$ is arbitrary, this implies the desired result.
	%	 In view of  and , we can follow t We shall omit the details.
\end{proof}
In view of Lemma \ref{linterp}, \eqref{rb1}, the classical Calder\'{o}n and Zygmund estimate, $\{\rt\}$
is bounded in $W^{2,q}(\Omega)$, where $q=\min\{2,\frac{Np}{N-p}\}$.
%In summary, claim
Passing to subsequences if necessary, we may assume
\begin{eqnarray}
\ut&\rightharpoonup & \mbox{$u$ weakly in $W^{1,p}(\Omega)$},\label{otn15}\\
\rt&\rightharpoonup  & \mbox{$\rho$ weakly in $W^{2,q}(\Omega)$ and a.e. on $\Omega$}.\label{otn13}
\end{eqnarray}
By virtue of Claim \ref{gup},
\begin{equation}
|\nabla\ut|^{p-2}\nabla\ut\rightharpoonup |\nabla u|^{p-2}\nabla u\ \ \mbox{weakly in $\left(L^{\frac{p}{p-1}}(\Omega)\right)^N$.}
\end{equation}
With the aid of Fatou's Lemma, we deduce from Claim \ref{lnr} that
$$\int_{\Omega}|\ln \rho| \, dx
\leq\liminf_{\tau\rightarrow 0}\int_{\Omega}|\ln \rt| \, dx\leq c.$$
Therefore, the set
$$A_0=\{(x)\in\Omega:\rho(x)=0\}$$
has Lebesque measure $0$. 
This combined with (\ref{otn13}) asserts that
\begin{equation}
\ln \rt\rightarrow \ln\rho \ \ \ \mbox{a.e. on $\Omega$}.
\end{equation}
Obviously, we have from \eqref{ue} that
\begin{equation}
\tau|\ut|^{p-2}\ut\rightarrow \mbox{$0$ strongly in $L^{\frac{p}{p-1}}(\Omega)$, and thus a.e on $\Omega$ }
\end{equation}
(passing to a subsequence if need be). Recall (\ref{ns11}) to obtain
\begin{equation}
-\plapt\rightarrow \ln\rho \ \ \ \mbox{a.e. on $\Omega$}.
\end{equation}
On the other hand, we conclude from Claim \ref{lnr} and \eqref{ue} that the sequence
$\{-\plapt\}$ is bounded in both $L^1(\Omega)$ and $\left(W^{1,p}(\Omega)\right)^*$. Hence we have
\begin{equation}
-\plapt\rightharpoonup -\plap\equiv\mu \ \ \ \mbox{weakly in both $\mathcal{M} (\overline{\Omega})$ and $\left(W^{1,p}(\Omega)\right)^*$}.\label{ow1}
\end{equation}
The key issue is: do we have
$$-\plap=\mu=\ln\rho?$$
The following claim addresses this issue.
\begin{clm}The restriction of $\mu$ to the set $\overline{\Omega}\setminus A_0$ is a function. This function is exactly $\ln\rho$. That is, the Lebesgue decomposition of
	$\mu$ with respect to the Lebesgue measure is $\ln\rho+\nu_s$, where $\nu_s$ is a measure supported in $A_0$, and we have
	\begin{equation}
	\rho=e^\mu \ \ \mbox{on the set $\overline{\Omega}\setminus A_0$.}
	\end{equation}
	That is, $\ln\rho$ is the function $g_a$ in the definition of a weak solution.
\end{clm}\begin{center}
	
\end{center}
\begin{proof} The proof is almost identical to the proof of Proposition 3.7 in \cite{LX}. For the reader's convenience, we shall reproduce it here.
	Keep in mind that since $\mu\in\left(W^{1,p}(\Omega)\right)^*$ each function in $W^{1,p}(\Omega)$ is $\mu$-measurable,  and thus it is well-defined except on a set of $\mu$ measure $0$. Furthermore,
	$ \langle \mu, v \rangle \, =\int_{\Omega} v \, d\mu$ 
	for each $v\in W^{1,p}(\Omega) $. 
	For $\varepsilon>0$ let $\te$ be a smooth function on $\mathbb{R}$ having the properties
	$$\te(s)=\left\{\begin{array}{ll}
	1 &\mbox{if $s\geq 2\varepsilon$,}\\
	0&\mbox{if $s\leq \varepsilon$ \hspace{.5in} and}
	\end{array}\right.$$
	$$
	0\leq\te\leq 1\ \ \mbox{on $\mathbb{R}$.}$$
	Then it is easy to verify from Claim \ref{rtc} that we still have
	\begin{equation}
	\te(\rt)\rightarrow\te(\rho)\ \ \mbox{strongly in $W^{1,p}(\Omega)$ for each $p\leq 2$.}\label{otn14}
	\end{equation}
	Pick a function $\xi$ from $C^\infty(\overline{\Omega})$. Multiply through (\ref{ns11}) by $\xi\, \te(\rt)$ and integrate the resulting equation over $\Omega$ to obtain
	\begin{equation}
	-\int_{\Omega}\plapt\te(\rt)\, \xi \, dx+\tau\int_{\Omega}|\ut|^{p-2} \ut\, \te(\rt)\, \xi \, dx
	=\int_{\Omega}\ln\rt \, \te(\rt)\, \xi \, dx.\label{otn20}
	\end{equation}
	For each fixed $\varepsilon$ we can infer from \eqref{rb1} that the sequence $\{\ln\rt \, \te(\rt)\}$ is bounded in $L^p(\Omega)$ for any $p>1$. This, along with (\ref{otn13}), gives
	$$\int_{\Omega}\ln\rt \, \te(\rt)\, \xi \, dx\rightarrow\int_{\Omega}\te(\rho)\ln\rho\, \xi \, dx.$$
	Observe from (\ref{otn14}) and (\ref{ow1}) that
	\begin{equation}
	-\int_{\Omega}\plapt \, \te(\rt)\, \xi\, dx
	=\langle -\plapt,\te(\rt)\, \xi \rangle\, 
	\rightarrow\int_{\Omega}\te(\rho)\, \xi \, d\mu.
	\end{equation}
	Taking $\tau\rightarrow 0$ in (\ref{otn20}) yields
	\begin{equation}
	\int_{\Omega}\te(\rho)\,\xi \,d\mu=\int_{\Omega}\te(\rho)\ln\rho \, \xi \, dx.\label{otn21}
	\end{equation}
	%Remember that
	%$\rho$ is $\mu$-measurable. That is to say,
	Obviously, $\rho\in W^{1,p}(\Omega)$, and thus it is well-defined except on a set of $\mu$ measure $0$. We can easily conclude from the definition of $\te$ that $\{\te(\rho)\}$ converges everywhere on the set where $\rho$ is defined as $\varepsilon\rightarrow 0$. With the aid of the Dominated Convergence Theorem, we can take $\varepsilon\rightarrow 0$ in (\ref{otn21}) to obtain
	$$\int_{\Omega\setminus A_0}\, \xi \, d\mu=\int_{\Omega\setminus A_0}\ln\rho \, \xi \, dx.$$
	This is true for every $\xi\in C^\infty(\overline{\Omega})$, which means
	\begin{equation}
	\mu=\ln\rho\ \ \ \mbox{on $\Omega\setminus A_0$.}
	\end{equation}
	This proves the claim.	\end{proof}
%$$	\mu=\ln\rho\ \ \ \mbox{on $\Omega\setminus A_0$.}$$
%	The proof is complete.

With this claim, the proof of Theorem \ref{th1.1} is now totally completed.
%	\end{proof}

%\end{proof}
\section{Remarks about the time-dependent problem}

In this section, we first fabricate an approximation scheme for the time-dependent problem \eqref{p1}-\eqref{p3}. This is based upon Theorem \ref{p21}. Then
we show that estimate \eqref{nm2} is preserved for the approximate problems.

Let $T>0$ be given. For each $j\in\{1,2,\cdots,\}$ we divide the time interval $[0,T]$ into $j$ equal subintervals. Set
$$\delta=\frac{T}{j}.$$
We discretize  \eqref{p1}-\eqref{p3} as follows. For $k=1,\cdots, j$, we solve recursively the system
\begin{eqnarray}
\frac{u_k-u_{k-1}}{\delta}-\Delta\esk+\delta\sk&=&0\ \ \ \mbox{in $\Omega$},\label{s31}\\
-\plapk+\delta|\uk|^{p-2}\uk&=&\sk\ \ \ \mbox{in $\Omega$},\label{s32}\\
\nabla\esk\cdot\nu=\nabla\uk&=&0\ \ \ \mbox{on $\partial\Omega$}.\label{s33}
\end{eqnarray}
Introduce the functions
\begin{eqnarray}
\utj(x,t)&=&\frac{t-t_{k-1}}{\delta}\uk(x)+\left(1-\frac{t-t_{k-1}}{\delta}\right)\uko(x), \ x\in\Omega,  \ t\in(t_{k-1},t _k],\\
\ubj(x,t)&=&\uk(x), \ \ \ x\in\Omega, \ \ t\in(t_{k-1},t_k],\\
\sbj(x,t)&=&\skx, \ \ \ x\in\Omega, \ \ t\in(t_{k-1},t_k],
\end{eqnarray}
where $t_k=k\delta$. We can rewrite \eqref{s31}-\eqref{s33} as
\begin{eqnarray}
\frac{\partial\utj}{\partial t}-\Delta \esbj+\delta\sbj&=&0\ \ \ \mbox{in $\omt$},\label{omm1}\\
-\plapj+\delta|\ubj|^{p-2}\ubj &=&\sbj \ \ \ \mbox{in $\omt$}.
\end{eqnarray}
We proceed to derive a priori estimates for the sequence of approximate solutions $\{\utj,\ubj,\sbj\}$. 

\begin{prop}
	There holds
	\begin{eqnarray}
	\lefteqn{\frac{1}{p}\max_{0\leq t\leq T}\int_\Omega|\nabla\ubj|^p dx +4\int_{\Omega_T}|\nabla e^{\frac{1}{2}\sbj}|^2dxdt}\nonumber\\
	&&+\frac{\delta}{p}\max_{0\leq t\leq T}\int_\Omega|\ubj|^pdx
	+\delta\int_{\Omega_T}\sbj^2 dxdt\nonumber\\
	&\leq &\frac{1}{p}\int_\Omega|\nabla u_0|^p dx+
	\frac{\delta}{p}\int_\Omega|u_0|^pdx.
	\end{eqnarray}
\end{prop}

Obviously, this proposition is the discretized version of \eqref{nm2}.

\begin{proof} 
	%Take the gradient of both sides of \eqref{s31}, then take the dot product of the resulting equation with $\nabla\uk$, and integrate over $\Omega$ to obtain
	Multiply through \eqref{s31} by $\psi_k$ and integrate to obtain
	\begin{equation}
	\int_\Omega\frac{ u_k- u_{k-1}}{\delta}\psi_k \, dx
	+\int_\Omega\nabla\left(\esk\right)\cdot\nabla\sk \, dx
	+\delta\int_\Omega\sk^2 \, dx=0.\label{ota10}
	\end{equation}
	The second integral in the preceding equation is computed as follows:
	\begin{eqnarray}
	\int_\Omega\nabla\left(\esk\right)\nabla\sk \, dx &=&
	\int_\Omega\esk \, |\nabla\sk|^2 \, dx\nonumber\\
	&=&4\int_\Omega|\nabla e^{\frac{1}{2}\psi_k} |^2dx.	\label{ota8}		
	\end{eqnarray}
	Using $\uk-\uko$ as a test function in \eqref{s32} yields
	\begin{eqnarray}
	\int_\Omega(u_k- u_{k-1})\psi_k \, dx&=&\io|\nabla\uk|^{p-2}\nabla\uk(\nabla\uk-\nabla\uko)dx\nonumber\\
	&&+\delta\io |\uk|^{p-2}\uk(\uk-\uko)dx\nonumber\\
	&\geq&\frac{1}{p}\io\left(|\nabla\uk|^p-|\nabla\uko|^p\right)dx+\frac{\delta}{p}\io\left(|\uk|^p-|\uko|^p\right)dx.\label{ota9}
	\end{eqnarray}
	The last step is due to (3) in Lemma \ref{elmen}.
	Substituting \eqref{ota8} and \eqref{ota9} into \eqref{ota10} yield
	\begin{eqnarray}
	\lefteqn{\frac{1}{p\delta}\int_\Omega\left(|\nabla\uk|^p-|\nabla\uko|^p\right)dx +4\int_\Omega|\nabla e^{\frac{1}{2}\sk}|^2dx}\nonumber\\
	&&+\frac{1}{p}\int_\Omega\left(|\uk|^p-|\uko|^p\right)dx
	%+\delta^2\int_\Omega\sk\uk dx\nonumber\\
	+\delta\int_\Omega\sk^2 dx\leq 0.
	\end{eqnarray}
	Then the proposition follows from multiplying through the above inequality by $\delta$ and summing up the resulting one over $k$.
	% Note that the last step in () is due to Theorem 3.1.
\end{proof}
Obviously, 	this theorem is not enough to justify passing to the limit in \eqref{omm1}. It remains open to find
additional estimates to accomplish the feat.

\bigskip
\noindent{\bf Acknowledgment.} The author is grateful to Prof. Jian-Guo Liu for originally
bringing this problem to his attention.

\end{document}